\documentclass[12pt]{article}
\usepackage{amsmath,latexsym,amssymb,amsthm,enumerate,amsmath,amscd}
\usepackage[mathscr]{eucal}
\usepackage[all,cmtip]{xy}
\theoremstyle{plain}
\newtheorem{theorem}{Theorem}[section]

\newtheorem{corollary}[theorem]{Corollary}
\newtheorem{lemma}[theorem]{Lemma}
\theoremstyle{definition}
\newtheorem{definition}[theorem]{Definition}

\newtheorem{example}[theorem]{Example}

\newtheorem{thm}{Theorem}

\newtheorem{rem}[thm]{Remark}

\newtheorem*{uthm}{Theorem}

\newtheorem*{ack}{Acknowledgment}

\newtheorem{quest}[thm]{Question} 
\numberwithin{equation}{section}
\numberwithin{table}{section} 
\DeclareMathOperator{\Grass}{\rm{Grass}}

\def\MR{\mathrm{MR}}

\def\Grass{\mathrm{Grass}}

\def\cod{\mathrm{cod}}
\def\codk{{\mathrm{cod}_{\sf k}}}
\def\<{\left<}
\def\>{\right>}
\def\Gl{\mathrm{Gl}}

\def\ns{\footnotesize \it}

\def\max{\mathrm{max}}

\def\dim{\mathrm{dim}}

\title{Strata of vector spaces of forms in $R={\sf k}[x,y]$, and of rational curves in $\mathbb P^{k}$.\footnote{ {\bf 2010 Mathematics Subject Classification}: Primary: 14D20; Secondary: 13D40, 13E10, 14H60.
{\bf Keywords}: rational curve, tangent bundle, parametrization.}}
\author{
Anthony Iarrobino\\[.05in]
{\ns Department of Mathematics, Northeastern University, Boston, MA 02115,
 USA.
}\\[.2in]}
\date{September 14, 2014}
\begin{document}
\maketitle
\begin{abstract}
Consider the polynomial ring $R={\sf k}[x,y]$ over an infinite field $\sf k$ and the subspace $R_j$ of degree-$j$ homogeneous polynomials. The Grassmanian $G=\Grass (R_j,d)$  parametrizes the vector spaces $V\subset R_j$ having dimension $d$. The strata $\Grass_H(R_j,d)\subset G$ determined by the Hilbert functions $H=H(R/(V))$ or, equivalently, by the Betti numbers of the algebras $R/(V)$, are locally closed and irreducible of known dimension. They satisfy a frontier property that the closure of a stratum is its union with lower strata \cite{I,I1}.  The strata are determined also by the decomposition of the restricted tangent bundle $\mathcal T_V=\phi_V^*(\mathcal T)$ where $\mathcal T$ is the tangent bundle on $\mathbb P^{d-1}$ and $\phi_V$ is the rational curve determined by $V$. Each stratum corresponds to a partition, and the the poset of strata under closure is isomorphic to the poset of corresponding partitions in the Bruhat order.  They are coarser than the strata defined by D. Cox, A. Kustin, C. Polini, and B.~Ulrich \cite{CKPU} and studied further in \cite{KPU}, that are determined in part by singularities of $\phi_V$. We explain these results and give examples to make them more accessible.  We also generalize a result of
D.~Cox, T. Sederberg, and F. Chen and another of C.~D'Andrea concerning the dimension and closure  of $\mu$ families of parametrized rational curves  from planar \cite{CSC,D} to higher dimensional embeddings (Theorem \ref{paramratcurvethm}).

\end{abstract}
Consider the polynomial ring $R={\sf k}[x,y]$ with maximum ideal ${\mathfrak m}=(x,y)$ over an infinite field $\sf k$, and write $R=\oplus_0^\infty R_j$. Let $V\subset R_j$ be a $d$-dimensional vector space of degree-$j$ forms; then $V$ determines a rational curve $\phi_V: \mathbb P^1\to \mathbb P^{d-1}$. We denote by $\Grass(R_j,d)$ the Grassmann variety parametrizing all such $d$-dimensional subspaces of $R_j$. The Hilbert function $H(A), A=R/(V)$ is the sequence 
\begin{equation}
H(A)=(1,2,\ldots ,j, j+1-d,  \ldots h_i,\ldots ), h_i=\dim_{\sf k}(A_i).
\end{equation} 
The sequences $H$ possible for such a Hilbert function are well-known (Lemma \ref{Hstratalem}).  We denote by $\Grass_H(R_j,d)\subset \Grass(R_j,d)$ the corresponding parameter space.\par
C. Polini described in her talk \cite{Po} her study with collaborators D. Cox, A. Kustin, and B.~Ulrich \cite{CKPU,KPU}  of certain locally closed  subfamilies  of $\Grass(R_j,d)$ determined by the singularity types of the associated rational planar and space curves $\phi_V$. They relate these singularity types to strata of normalized relation matrices for the ideal $(V)$ by relation degrees, then by zeroes of entries and equalities between entries. The purpose of this note is to recall and illustrate some of the properties of the coarser stratification by Hilbert function $H(R/(V))$ -- or relation degrees of $(V)$ -- studied in \cite{I,GIS,I1}. We hope that others might wish to extend these combinatorial descriptions and deformation properties to the finer stratification described by C. Polini in her talk. In section~\ref{paramratcurves} we generalize a result of D.~Cox, T. Sederberg, and F. Chen \cite{CSC}, and of C.~D'Andrea \cite{D} on $\mu$-bases to higher embedding dimensions.
In Section \ref{ratscrollsec} we briefly mention analogous results for rational scrolls.\par
 We denote by $\mathcal T$ the tangent bundle on $\mathbb P^{d-1}$. The restricted tangent bundle $\mathcal T_V=\phi_V^*(\mathcal T)$ on $\mathbb P^1$ has a decomposition $\mathcal T_V=\oplus \mathcal O(k_i)$ into a direct sum of line bundles: knowing the degrees $K=\{ k_i\}$  of the summands is equivalent to knowing the Betti numbers in a minimal resolution of a basis of $V$, or equivalently, to knowing the column degrees $D$ of the $d\times (d-1)$ Hilbert-Burch relation matrix $M$ among the generators. This decomposition is also equivalent to simply knowing the Hilbert function of the algebra  $R/(V)$ \cite{Asc1,Asc2, GIS, GHI, Ber}. \par
There have been many further studies of the geometry of this decomposition, as \cite{BR,Ber1,Ber,Cl,EV1,EV2,G,GS,GHI,He,HK,Ka,PS,Ra1,Ra2,Ra3, Ran1,Ran2}. \par In \cite[Section 4B]{I} and \cite{I1} we studied the stratification of $G=\Grass(R_j,d)$ by the Hilbert function $H=H(R/(V))$. We determined which sequences $H$ can occur, and gave a 1-1 correspondence $H\to P_H$ between the Hilbert function strata $\Grass_H(R_j,d)$ and certain partitions of $j+1-d$ \cite[Lemma 2.23]{I1}\footnote{Warning: the $H$ used here is written $T$ in \cite{I1}; it is the \emph{tail} (higher degree) part of the Hilbert function of the full ancestor ideal, for which there are similar results.} (see Theorem \ref{mainthm1} below). We also determined the codimension of the $H$ stratum in $G$ \cite[Proposition 4.7]{I},\cite[Thm. 2.17]{I1} (see \eqref{codeq2} below). \par  We denote by $G_H$ the projective variety parametrizing the graded ideals $I\subset R$ such that $H(R/I)=H$. We set $c_H=\lim_{i\to \infty}h_i$: it is the degree of the common factor of each ideal $I\in G_H$.
Given a sequence $H=(\cdots h_i\cdots)$ let $\varrho(H)=\min\{ i\mid h_i\not=i+1\}$, and set $E_H=(e_{\varrho},\ldots , e_i,\ldots  ), e_i=\Delta H_{i-1}=h_{i-1}-h_{i}$. Then $H$ occurs as a Hilbert function $H(R/I) $ for a graded ideal $I$ of $R$ iff $e_i\ge 0$ for $i\ge \varrho (H)$ \cite{Mac}.  We showed
\begin{theorem}\label{Hdimthm}\cite[Theorem~2.12]{I},
\cite[Theorem 1.10]{I1}. 
The variety $G_H$ when nonempty is an irreducible nonsingular projective variety of dimension
\begin{equation}\label{dimeq}
\dim G_H=c_H+\sum_{i\ge \varrho (H)} (e_i+1)\cdot (e_{i+1}).
\end{equation}
\end{theorem}
Now we set $\varrho=j$, so $E_H=(e_j,e_{j+1},\ldots), e_i=h_{i-1}-h_i$, and let $\tau_H =e_{j+1}+1$. We denote by $\Grass_H(R_j,d)$ the subvariety of $G=\Grass(R_j,d)$ parametrizing the vector spaces $V\subset R_j$ of Hilbert function $H(R/(V))=H$. It is locally closed in $G$.
\begin{lemma}\label{Hstratalem}\cite[Proposition 4.6]{I}
The stratum $\Grass_H(R_j)\subset G$ is nonempty if and only if $1\le \tau_H\le \min (d,j+2-d)$, and $E_H$ is non-increasing. Then it is irreducible and open dense in $G_H$.
\end{lemma}
We denote by $\mathcal H(j,d)$ the set of sequences $H=(h_j,h_{j+1},\ldots)$ satisfying the conditions of  Lemma \ref{Hstratalem}.
 For $H',H\in \mathcal H(j,d)$ we define the termwise partial order 
\begin{equation}\label{Hposeteq}
 H'\ge H\Leftrightarrow h'_i\ge h_i\text { for } j\le i< \infty. 
\end{equation}
  These strata satisfy a frontier property that the closure of a stratum is a union of strata.
\begin{theorem}\label{closurethm}\cite[Theorem 4.10]{I},\cite[Theorem 2.32]{I1}  The closure $\overline{\Grass_H(R_j,d)}$ of each stratum of $G$ satisfies
\begin{equation}\label{frontiereq}
\overline{\Grass_H(R_j,d)}=\bigcup_{H'\in \mathcal H(j,d) \mid H'\ge H}\Grass_{H'}(R_j,d)
\end{equation}
the union of the strata whose Hilbert function is greater or equal termwise to $H$. The projection $\mathfrak p:\,G_H\to \overline{Grass_H(R_j,d)}$ where $ \mathfrak p(I)= I_j$ is a desingularization of the Zariski closure $\overline{\Grass_H(R_j,d)}\subset G$.
\end{theorem}
\par  Showing that $\mathfrak p$ is surjective involves constructing for each  given ideal $I'=(V')$ of Hilbert function $H'=H(R/I')$ a new ideal $I\supset I'$ such that  $H(R/I)=H$; this is accomplished step by step using properties of the invariant $\tau (V)$ of \eqref{taueq} below \cite[Proposition 4.9]{I},\cite[Lemma 2.30]{I1}. The frontier property \eqref{frontiereq} then follows from the known non-singularity and irreducibility of $G_H$ (Theorem \ref{Hdimthm}).\par
We here relate these results to the partitions determining the minimal resolution of $A=R/(V)$ and give examples that we hope will make these results more accessible. 
\section{Restricted tangent bundle strata of $\Grass(R_j,d)$.}
We first state some relevant results of \cite{I1} in Theorem \ref{mainthm1}. We then illustrate these results by two examples: $(j,d)=(6,3)$ (Example \ref{6,3ex}) and $(j,d)=(8,3)$ in Example \ref{8,3ex}. The latter is the example C.~Polini gave in her talk, that of a three-dimensional vector space\ $V=\langle f_1,f_2,f_3\rangle, f_i\in R_8$ with no common factor (base point).  In her discussion the singularity type corresponds to several invariants of the $3\times 2$ ``Hilbert-Burch'' matrix $M$ of relations among the generators of $I=(V)=(f_1,f_2,f_3)$: these include the set of zero entries and also the equalities between entries in a normal form for $M$, and as well more subtle invariants having to do with factors of the entries.\par For $(j,d)=(8,3)$, when there are no base points, the possibilities for the column degrees of $M$ are   $(3,3), (4,2)$ and $(5,1)$. C.~Polini in her talk discussed the behavior under deformation of singularity types within the first,  generic family of $(3,3)$ relation degrees. Our approach applies to all pairs of positive integers $(j,d), d\le j+1$ and to all sequences $D$  but deals with the coarser invariant $H$. We begin by describing properties of the $\tau$ invariant: as we shall see, $d-\tau (V)$ is just the number of relations of $(V)$ in degree $j+1$, so is equal to the number of $1's$ in the relation degrees $D$ (equation \eqref{minreseq}). \par 
\begin{definition} Let $V\subset R_j$ be a vector subspace. We denote by $\overline{V}$ the \emph{ancestor ideal} of $V$:
\begin{equation}\label{anceq}
\overline{V}=V:R_j+\cdots +V:R_1+(V),
\end{equation}
where $V:R_k=\{ f\in R_{j-k}\mid R_k\cdot f\subset V\}$. It satisfies 
\begin{equation}\label{anc2eq}
\overline{V}\cap \mathfrak m^j=(V)
\end{equation}
and is the largest ideal satisfying \eqref{anc2eq}.\footnote{The definition of ${\overline{V}}$  is independent of the number of variables; but the results concerning $\tau(V)$ require the ring $R={\sf k}[x,y]$.}
The $\tau$ invariant of $V$ is
\begin{align}\label{taueq}
\tau (V)&=\dim_{\sf k} R_1\cdot V-\dim_{\sf k} V\\
&=1+ \cod_{\sf k} V-\cod_{\sf k} R_1V=1+e_{j+1}\notag
\end{align}
where $\cod_{\sf k} V=j+1-d$ and $\cod_{\sf k} R_1V=\dim_{\sf k} R_{j+1}/R_1V=H(R/(V))_{j+1}$.
\end{definition}

\begin{lemma}\cite[Lemma 2.2]{I1} The integer $\tau(V)$ is the number of generators of the ancestor ideal $\overline{V}$. We have 
\begin{equation}\label{boundtaueq}
1\le \tau(V)\le \min\{\dim_{\sf k} V, 1+\cod_{\sf k} V\}. 
\end{equation}
\end{lemma}
\begin{example} For $(j,d)=(j,3)$ we have $1\le \tau (V)\le 3$: 
\begin{align}
\tau (V)=1&\Leftrightarrow V=\langle x^2,xy,y^2\rangle \cdot f \text { where  } f\in R_{j-2};\notag\\
\tau (V)\le 2&\Leftrightarrow V=\langle x\cdot f,y
\cdot f,f_3\rangle \text { where } f\in R_{j-1};\notag\\
\tau (V)=3&\text { otherwise }.\notag
\end{align}
\end{example}
We denote by $\Grass_\tau (R_j,d)$ the subfamily of $\Grass (R_j,d)$ parametrizing vector spaces with $\tau (V)=\tau$. It is the subfamily of those $V$ with $d-\tau$ linear relations (see \eqref{Deq}).
\begin{lemma}\cite[Theorem 2.17]{I1} The codimension of $\Grass_\tau (R_j,d)$  in $\Grass(R_j,d)$ satisfies 
\begin{equation}\label{codtaueq}
\cod\, \Grass_{\tau}(R_j,d)=(\dim_{\sf k} V-\tau )(\cod_{\sf k} V- (\tau -1))=(d-\tau)(j+2-d-\tau).
\end{equation}
\end{lemma}
For an integer $n\in \mathbb Z$ we denote by $|n|^+$ the integer $n$ if $ n\ge 0$ and $0$ otherwise. We write $P\vdash n$ for ``$P$ partitions $n$''. Recall the \emph{Bruhat} or \emph{orbit closure} partial order on partitions of $n$

\begin{equation}\label{Bruhateq}  P\ge P'\quad\Leftrightarrow\quad \forall k\in \mathbb N,\,\, \quad\sum_{0}^kP_i\ge 
\sum_{0}^kP'_i.
\end{equation}
We denote by $\mathcal P(a)$ the poset of partitions of $a$ and by $\mathcal P(a,b)$ the poset of partitions of $a$ into exactly $b$ non-zero parts, under the Bruhat partial order; set $\mathcal P(a,\le b)= \mathcal P(a,1)\cup \cdots \cup \mathcal P(a,b)$. Recall that  $c_H, H={H(R/(V))}$ is the degree of gcd$(V)$  -- the number of base points.
\begin{lemma}\cite[Lemma 2.23,(2.45)]{I1} The minimal resolutions that occur for algebras $R/(V)$ for which $\tau (V)=\tau$ and $c_H=0$ can be written 
\begin{equation}\label{minreseq}
0\to \sum_{i=1}^{\tau -1} R(-j-1-\lambda_i)\oplus R(-j-1)^{d-\tau}\to R(-j)^d\to R\to R/(V)\to 0,
\end{equation}
and correspond 1-1 to the partitions 
\begin{equation}\label{parteq}
\lambda=(\lambda_1,\ldots, \lambda_{\tau -1}), \ \lambda_1\ge \lambda_2\ge \cdots \ge \lambda_{\tau-1}>0 \text { of } j+1-d \text { into } \tau -1 \text { parts }.
\end{equation}
\end{lemma}\noindent
We define the partition 
\begin{equation}\label{Deq}
D=D(\lambda)=(\lambda +\underline{1},1^{d-\tau})=(\lambda_1+1,\ldots , \lambda_{\tau -1}+1, \overbrace{1,\ldots, 1}^{\quad d-\tau\quad}) 
\end{equation}
of $j$ into $d-1$ parts, which gives the complete set of relation degrees in \eqref{minreseq}, relative to $j$. Thus, $D$ is the column degrees of the Hilbert-Burch matrix $M_I$ used in \cite{Po} and determines a Hilbert function $H( D)$. We denote by $\MR_D(R_j,d)=\Grass_{H(D)}(R_j,d)$ the stratum of $\Grass(d,R_j)$ parametrizing $V$ satisfying \eqref{minreseq}, with $D$ from \eqref{parteq}, \eqref{Deq}. When $c_H\not=0$ there is a projection $V\to V:f$ where $f=\gcd (V)$. Then,
$\lambda $ partitions $j+1-c_H-d$ into $\tau-1$ parts and $D$ partitions $j-c_H$ into $d-1$ parts. Note that $\tau(H)=1\Leftrightarrow c_H=j+1-d\not=0.$
We denote by $\lambda^\vee$ the conjugate partition to $\lambda$ -- switch rows and columns in the Ferrers diagram of $\lambda$.
\begin{theorem}\label{mainthm1} \cite[Section 2.2, Lemma 2.23, Theorem 2.24,  (2.48),(2.50), (2.54)]{I1}. 
The relation between the partition $\lambda$ and the corresponding Hilbert function $H_\lambda=H(R/(V))$ is 
\begin{equation}\label{hilbparteq}
H_\lambda: \quad H_{\lambda ,{j+i}}=c_H+\sum_u |\lambda_u-i|^+ \text { for $i\ge 0$; also $\Delta H_{\ge j}=\lambda^\vee.$}
\end{equation}
We have for $c_{H_\lambda}=c_{H_\lambda'}=0$,
\begin{equation}\label{ordereq}
H_\lambda\le H_{\lambda'}\Leftrightarrow \lambda\le \lambda'\Leftrightarrow \lambda^\vee\ge (\lambda')^\vee\Leftrightarrow D(\lambda)\le D(\lambda')
\end{equation}
 in the Bruhat order of partitions. \par\noindent
The stratum $\MR_D(R_j,d)$is irreducible and has codimension in $G=\Grass(R_j,d)$, 
\begin{equation}\label{codeq2}
\ell(D)= c_{H(D)}\cdot (d-1)+\sum_{u\le v}(D_u-D_v-1)^+.
\end{equation}
Its codimension in $\Grass_{\tau}(R_j,d)$ is
$
\ell(\lambda)=c_{H(D)}\cdot (d-1)+\sum_{u\le v}(\lambda_u-\lambda_v-1)^+
$.\par

\end{theorem}
 Theorem \ref{closurethm} and \eqref{ordereq} of Theorem \ref{mainthm1} imply
\begin{corollary}\label{posetcor} Fix $(j,d)$. The poset of closures $\{\overline{\Grass_{H_\lambda}(R_j,d)}\mid \lambda$ from \eqref{parteq}$\}$  under inclusion is the poset of those Hilbert functions of Lemma \ref{Hstratalem} with $c_H=0$ under the termwise partial order. It is isomorphic under the map $H_\lambda\to D(\lambda)$ to the poset $\mathcal P(j,d-1)$  and is isomorphic under the map $H_\lambda\to \lambda$ to the poset $\mathcal P(j+1-d,\le (d-1))$.
\end{corollary}

\begin{example}\label{6,3ex} (As in C. Polini's talk). Let $(j,d)=(6,3)$, so $\cod_{\sf k} (V)=6+1-3=4$, and $\dim \Grass(R_6,3)=3\cdot 4=12$. By equation \eqref{boundtaueq} we have $1\le \tau(V)\le 3$. Assume first that there are no base points: $c_V=0$. When $\tau=3$ the Hilbert functions $H_\lambda$ are determined by the partitions of $\cod_{\sf k} V=4$ into $\tau-1=2$ parts. The generic case has partition $\lambda =(2,2)$, and corresponds to relation degrees $D=(3,3)$ and to Hilbert function $(H_\lambda)_{\ge 6}=(4,2,\underline{0})$. The special case for $\tau=3$ has partition $\lambda'=(3,1)$, and corresponds to relation degrees $D'=(4,2)$ and to $(H_{\lambda'})_{\ge 6}=(4,2,1,\underline{0})$. By \eqref{codeq2}  $\Grass_{H_{\lambda'}}(R_j,d)$ has codimension 1 in $\Grass(R_6,3)$. When $\tau =2$ the unique partition is $\lambda''=(4) $: the relation degrees $D''=(5,1)$, the Hilbert function $(H_{\lambda''})_{\ge 6}=(4,3,2,1,\underline{0})$; and by \eqref{codtaueq} $\Grass_{H_\lambda''}(R_6,3)=\Grass_2(R_6,3)$ has codimension 3 in $\Grass (R_6,3)$. Here $D''$ is the most special case with $c_H=0$.
\begin{table}
\begin{center}$
\begin{array}{|ccccccc|}
\hline
\lambda &c_H&\tau& H_{\ge 6}&\lambda^\vee=\Delta H_{\ge 6}&\dim&\cod\,\\
\hline
&&&&&&\\
(2,2)&0\,&3&(4,2,\underline{0})&(2,2)&12&0\\
(3,1)&0&3&(4,2,1,\underline{0})&(2,1,1)&11&1\\
(4)&0&2&(4,3,2,1,\underline{0})&(1,1,1,1)&9&3\\\hline
&&&&&&\\
(2,1)&1&3&(4,2,\underline{1})&(2,1)&10&2\\
(3)&1&2&(4,3,2,\underline{1})&(1,1,1)&8&4\\
(1,1)&2&3&(4,\underline{2})&(2)&8&4\\
(2)&2&2&(4,3,\underline{2})&(1,1)&7&5\\
(1)&3&2&(4,\underline{3})&(1)&6&6\\
(0)&4&1&(\underline{4})&(0)&4&8\\
\hline
\end{array}
$
\end{center}
\caption {The $\Grass_H(R_6,3)$ strata.}\label{HF1table}
\end{table}

To summarize, an open dense $U\subset\Grass(R_6,3)$ parametrizes $V$ having no base points, and is decomposed into three $H$-strata corresponding to the three partitions $\lambda,\lambda',\lambda''$ of four or, equivalently by Corollary \ref{posetcor}, the partitions $D,D'$ and $D''$ of six into 2 parts.\par
When $c_H=1$ there is a single base point: then $V=\ell \cdot V'$ where $\ell\in R_1$ and $ V'\in \Grass(R_5,3)$ has codimension $3$ in $R_5$. The two strata correspond to the two partitions of $3=\cod_{\sf k} V'$ into at most $\codk -1=2$ parts: so $\tau=3: \lambda'''=(2,1)$ and $(H_{\lambda'''})_{\ge 6}=(4,2,\underline{1})$; and $\tau=2: \lambda^{(4)}=(3)$ and
$(H_{\lambda^{(4)}})_{\ge 6}=(4,3,2,\underline{1})$. \par When $c_H=2$ the two Hilbert functions are $H_{\ge 6}=(4,\underline{2})$ for $\tau=3$ and partition $\lambda =(1,1)$ and $H_{\ge 6}=(4,3,\underline{2})$ for $\tau=2$ and $\lambda =(2)$. There are two more (see Table \ref{HF1table}). Theorem~\ref{closurethm} and a comparison of the Hilbert functions show that there are two maximal chains in $\mathcal H(6,3)$,  between $H_{\ge 6}=(4,3,\underline {2})$ and $(4,2,1,\underline{0})$, one corresponding to the chain  $\lambda^\vee = (1,1)\le (1,1,1)\le (1,1,1,1)\le (2,1,1)$ and the other to the chain $\lambda^\vee=(1,1)\le (2)\le (2,1)\le (2,1,1)$.
\end{example}
\begin{example}\label{8,3ex} For $(j,d)=(8,3)$ we have $\tau \le 3$ so for $c_H=0$ the partitions $\lambda$ are $(3,3)\le (4,2)\le (5,1)\le (6)$ with the $(3,3)$  stratum being open-dense. 
The Hilbert functions $H(R/(V))_{\ge 8}$ are specified in Table \ref{HF2table}.
Here the conjugate partition $\lambda^{\vee}$ satisfies
\begin{equation}
\lambda^\vee=\Delta H_{\ge 8}=(e_9(H),e_{10}(H),\ldots )
\end{equation}
 the first differences of $H_{\ge 8}$. Since $\Delta H_8=e_9=2$, we have for the $\lambda=(5,1)$ stratum $H_{\ge 8}=(6,4,3,2,1,0)$ that $\lambda^\vee=(2,1,1,1,1)$ and by \eqref{dimeq},
\begin{equation*}
\dim \Grass_H(R_8,3)=(3\cdot 2+3\cdot 1+2\cdot 1+2\cdot 1+2\cdot 1)=15.
\end{equation*}

The closure of the $\lambda =(5,1)$ stratum is its union with the eleven strata whose Hilbert functions $H_{\ge 8}$ lie above $(6,4,3,2,1,0)$: we have indicated these strata with $\ast$ in the rightmost column of Table \ref{HF2table}. By Theorem \ref{mainthm1} this closure has desingularization the family $G_H, H=(1,2,3,4,5,6,7,8,6,4,3,2,1,0)$  -- the family of all graded ideals in $R$ satisfying $H(R/I)=H$. 
\end{example}
\begin{example} For $(j,d)=(9,4)$ the poset of $H$-strata with $c_H=0$ under closure is by Corollary~\ref{posetcor} isomorphic to the poset $\mathcal P(6,\le 3)$ of partitions $\lambda$ of 6 into at most 3 parts, or, equivalently, to the poset $\mathcal P(9,3)$ of partitions $D$ of 9 into exactly 3 parts. The strata for $\lambda=(4,1,1)$ and for $\lambda'=(3,3)$ are incomparable: $H_{\ge 9}(\lambda)=(6,3,2,1,0)$ and $H_{\ge 9}(\lambda')=(6,4,2,0)$. By Theorems~\ref{closurethm} and \ref{mainthm1} the incomparability of partitions  $\lambda,\lambda'$ in the Bruhat order is equivalent to incomparablity of the Hilbert functions $H_\lambda, H_{\lambda'}$ in the termwise partial order, and implies that no closure of a subfamily of one of the strata can intersect the other stratum.
\end{example}
\begin{table}
\begin{center}$
\begin{array}{|cccccc|}
\hline
\lambda&c_H& H_{\ge 8}&\lambda^\vee=\Delta H_{\ge 8}&\dim&\ast \\
\hline
&&&&&\\
(3,3)&0&(6,4,2,0)&(2,2,2)&18&\\
(4,2)&0&(6,4,2,1,0)&(2,2,1,1)&17&\\
(5,1)&0&(6,4,3,2,1,0)&(2,1,1,1,1)&15&*\\
(6)&0&(6,5,4,3,2,1,0)&(1,1,1,1,1,1)&13&*\\\hline
&&&&&\\
(3,2)&1&(6,4,2,1,\overline{1})&(2,2,1)&16&\\
(4,1)&1&(6,4,3,2,1,\overline{1})&(2,1,1,1)&14&*\\
(5)&1&(6,5,4,3,2,1,\overline{1})&(1,1,1,1,1)&12&*\\
(2,2)&2&(6,4,2,\overline{2})&(2,2)&14&\\
(3,1)&2&(6,4,3,2,\overline{2})&(2,1,1)&13&*\\
(4)&2&(6,5,4,3,2,\overline{2})&(1,1,1,1)&11&*\\
(2,1)&3&(6,4,3,\overline{3})&(2,1)&12&*\\
(3)&3&(6,5,4,3,\overline{3})&(1,1,1)&10&*\\
(1,1)&4&(6,4,\overline{4})&(2)&10&*\\
(2)&4&(6,5,4, \overline{4})&(1,1)&9&*\\
(1)&5&(6,5,\overline{5})&(1)&8&*\\
(0)&6&(6,\overline{6})&(0)&6&*\\
\hline
\end{array}
$
\end{center}
\caption {The $\Grass_H(R_8,3)$ strata.}\label{HF2table}
\end{table}
\subsection{Parametrization of rational curves}\label{paramratcurves}
David Cox after seeing an earlier version of this note remarked on the connections to his work with T. Sederberg and F. Chen. He writes that ``the paper \cite{CSC} defines $\mathcal P_n$ as the set of all relatively prime
triples $(a,b,c)$ with $a,b,c \in {\sf k}[t]$ and $ n =\max (\deg(a), \deg(b),
\deg(c))$. Thus, if we homogenize with respect to a new variable $u$,
then $\mathcal P_n$ is canonically isomorphic to the set
\begin{equation}
\{(A,B,C) \in {\sf k}[t,u]^3 \mid A, B, C \text { homogeneous of degree $n$ and }
gcd(A,B,C) = 1  \}
\end{equation}
 Then the subset $\mathcal P_n^\mu$ is naturally isomorphic to the subset of
$\mathcal P_n$ consisting of triples $(A,B,C)$ for which the ideal $I = (A,B,C)$ has a free resolution, taking here $R ={\sf k}[t,u]$,
\begin{equation}
0 \to R(-n-\mu)+R(-2n +\mu) \to R(-n)^3 \to  0.
\end{equation}
with $\mu \le n-\mu$.''  Their main result is the irreducibility and dimension of the closure $\overline{\mathcal P}^\mu_n$:
\begin{uthm}{\cite[Theorem 1.1]{CSC}} For each $\mu, 0\le \mu\le \lfloor n/2\rfloor $ the closure $\overline{\mathcal P}_n^\mu$ is irreducible of dimension 
\begin{equation}\label{Pmueq1}
\dim \overline{\mathcal P}_n^\mu =\begin{cases} &3n+3 \text {\qquad if \quad$\mu=\lfloor n/2\rfloor$}\\
& 2n + 2\mu + 4 \text { \quad if \quad $\mu <\lfloor n/2\rfloor$ }
\end{cases}
\end{equation}
\end{uthm}
Carlos D'Andrea then gave an explicit approximation of $\mathcal P_n^{\nu}$ by a sequence in $\mathcal P_n^{\nu+1}$, showing
\begin{uthm}\cite[Thm. 1.2, (2)]{D} For $0\le \mu\le \lfloor n/2\rfloor
$ the closure  $\overline{\mathcal P}_n^\mu$ in $\mathcal P_n$ satisfies\begin{equation}
\overline{\mathcal P}_n^\mu=\bigcup_{\nu\le \mu}{\mathcal P}_n^\nu.
\end{equation}
\end{uthm}
We now compare to our notation: given $(n,\mu)$ the relation degree  partition is $D=(d_1,d_2)=(n-\mu, \mu)$, so by \eqref{codeq2} the codimension of the stratum $\MR_D(R_j,d)$ is $(n-2\mu-1)^+$: this is the codimension $3n+3-(2n+2\mu+4)$ given in \eqref{Pmueq1}.  Next, we generalize.
\par The affine space $\mathbb A^N,N=d(n+1) $ parametrizes $d$-tuples of degree-$n$ polynomials in ${\sf k}[x]$ and an open dense $\mathbb A_U^N$ parametrizes those that form linearly independent sets.
Given a partition $D$ of an integer $k, d-1\le k\le n$ into $d-1$ parts, we define a subscheme of $\mathbb A_U^N$,
\begin{equation}
\mathcal {CP}_n^D=\{(\mathcal A=(a_1,a_2,\ldots ,a_d),\,a_i\in {\sf k}[x]\}\subset \mathbb A_U^N, N=d(n+1) ,
\end{equation} comprised of those ordered sequencees of  $d$ linearly independent polynomials in ${\sf k}[x]$, each of degree less or equal $n$, that parametrize rational curves in $\mathbb P^{d-1}$ whose relation degrees are $D$. Here $k=n-c$ where $c$ is the degree of gcd($\mathcal A$). When $D$ partitions $n$ (we write $D\vdash n$) the open dense subscheme $\mathcal P_n^D\subset \mathcal {CP}_n^D$ parametrizes those $d$-tuples that are relatively prime and have highest degree $n$; we set $\mathcal P_{n,d}=\cup_{D\vdash n} \mathcal P_n^D$. The general linear group $\Gl_d(\sf k)$ acts freely on 
$\mathcal {CP}_n^D$ and we have a surjective projection with fibre $\Gl_d(\sf k)$

 \begin{equation}
\pi: \mathcal {CP}_n^D\to \MR_D(R_j,d): \pi (\mathcal A)=\langle  A\rangle.
\end{equation}
\begin{theorem}\label{paramratcurvethm} Let $D$ partition $n-c$ into $d-1$ non-zero parts. Then $\mathcal {CP}_n^D$ is an irreducible locally closed subvariety of $\mathbb A^N$ having dimension  (where $\ell(D)$ is from \eqref{codeq2})
\begin{equation}\label{paramratcurve2eq}
N-\ell(D).
\end{equation}
The closure $\overline{\mathcal {CP}}_n^D$ in  $\mathbb A_U^{(n+1)d}$ satisfies 
\begin{equation}
\overline{\mathcal {CP}}_n^D=\bigcup_{H(D')\ge H(D)}\overline{\mathcal {CP}}_n^{D'}.
\end{equation}
For $D\vdash n$ the closure of ${\mathcal {P}}_n^D$ in $\mathcal P_{n,d}$ is $\overline{\mathcal {P}}_n^D=\bigcup_{D'\ge D, D'\vdash n}\overline{\mathcal {P}}_n^{D'}$.
\end{theorem}
\begin{proof} This is a consequence of Theorems \ref{closurethm} and \ref{mainthm1},  and the property of the fibration $\pi$, that preserves codimensions and closures.
\end{proof}\par
There are now many articles on $\mu $ bases, as \cite{SC,WJG} for space curves. We hope this extension to all embedding dimensions might be useful.
\section{Rational scroll strata for vector spaces of forms.}\label{ratscrollsec}
Recall that the ancestor ideal $\overline V$ defined in \eqref{anceq} is the largest ideal satisfying $\overline{V}\cap \mathfrak m^j=(V)$. We define
the ``nose' or ``scroll'' Hilbert function $N(V)$:
\begin{equation}
N(V)= H_{\le j} (R/\overline{V}),
\end{equation}
which satisfies $e_{i-1}(N)\le e_i(N)$ for $i\le j$. Let $\mathcal N(j,d)$ denote the set of possible nose sequences.
Here $N$ is determined by a partition ${\mathcal A}_N$ of $d$ into $\tau$ parts. Setting $I=\overline V$, we have
\begin{equation}\label{noseeq}
\dim_{\sf k}I_{j-i}=\sum_u\mid a_u-i\mid^+.
\end{equation}
The partition invariant $\mathcal A_V={\mathcal A}_{N(V)}$ determines the mininum rational scroll on which the rational curve determined by $V$ lies.  Here $N'$ more special is equivalent to $N'\le N$ and $A'\ge A$ \cite[Definition 1.14, Lemma 2.28]{I1}.  We have \cite[Lemma 2.30, Theorem 2.32]{I1} the analogue for $N$ of Theorem \ref{closurethm},
\begin{equation}
\overline{\Grass_N(R_j,d)} =\bigcup_{N'\in \mathcal N(j,d), N'\le N} \Grass_{N'}(R_j,d).
\end{equation}
The codimension of $\Grass_ {N_{\mathcal A}}(R_j,d)$ in $\Grass_\tau (R_j,d)$ satisfies (ibid. Theorem 2.24)
\begin{equation}
\cod_{\negthickspace\tau}\, \Grass_ {N_A}(R_j,d)=\ell (A)=\sum_{u\le v}(a_u-a_v-1)^+.
\end{equation}
We showed analogues in \cite{I1} for $N_A$ and also for the pair $(N_A,H_{\lambda})$ of each statement in Theorem~\ref{mainthm1}. We state the codimension formulas in $\Grass_\tau(R_j,d)$ since the $\Grass_ {N_A}(R_j,d)$ and $\Grass_ {H_V}(R_j,d)$ strata intersect properly in their corresponding $\Grass_\tau(R_j,d)$!
\begin{example} Let $j=9, d=4$, and fix $\tau =3$.
Consider the partition $A=(2,1,1)$, where $\dim_{\sf k}  \overline{V}_{6,7,8,9}=(0,0,1,4)$ and $N_{V}= (1,2,3,4,5,6,7,8,8,6)$.
The closure of the stratum $N_A$ is all strata $N_{A'}
$ that are termwise smaller or equal to $N_A$: such that $A'\ge A$ in the orbit closure order of partitions of $4$. A generic element $V_{f,W}$ of $\Grass_{N_A}(R_9,4)$ is determined by a pair $(f,W)\mid f\in R_8$ is generic, and $W$ is a generic $2$-dimensional subspace of $\langle R_9/(R_1\cdot f)\rangle$: so $\Grass_{N_A}(R_9,4)$ is fibred over $\mathbb P_8=\mathbb P (R_8)$  by an open in the Grassmanian $\Grass (8,2)$. Its dimension is $8+12=20$, which agrees with the dimension of $\Grass_{3}(R_9,4)$ from \eqref{codtaueq}. Similar arguments give the structure of other nose strata of $\Grass(R_9,4)$.
\end{example}
\begin{table}
\begin{center}$
\begin{array}{|ccccccc|}
\hline
A&\tau&c_H& H_{6,7,8,9}&\dim_{\sf k}\overline{V}_{6,7,8,9}&A^\vee&\dim \\
\hline
&&&&\\
(1,1,1,1)&4&0&(7,8,9,6)&(0,0,0,4)&(4)&24\\
(2,1,1)&3&0&(7,8,8,6)&(0,0,1,4)&(3,1)&20\\
(2,2)&2&0&(7,8,7,6)&(0,0,2,4)&(2,2)&14\\
(3,1)&2&0&(7,7,7,6)&(0,1,2,4)&(2,1,1)&13\\
(4)&1&6&(6,6,6,6)&(1,2,3,4)&(1,1,1,1)&6\\

\hline
\end{array}
$
\end{center}
\caption {Nose Hilbert function $\Grass_{ N_A}(R_9,4)$ strata.}\label{HFtable}
\end{table}
\subsection{Problems}
We pose some open questions that warrant further exploration. 
\begin{quest} How do the frontier and desingularization properties of the stratification of $\Grass(R_j,d)$ by $\{\Grass_H(R_j,d)\}$ extend to the finer stratification of rational curves by singularity types? 
\end{quest}

\begin{quest} The behavior of the closures $\overline{\Grass_H(R_j,d)}$ as the limit vector spaces pick up base points is described in Theorems \ref{closurethm} and \ref{mainthm1}. How does this extend to the finer stratification? 
\end{quest}
\begin{quest}  The variety structure of the closures $\overline{\Grass_H(R_j,d)}$ is not known; these are not Schubert varieties \cite{I} (the count is wrong) but are they related to Schubert varieties? What are their cohomology classes in $G$? 
We know $\Grass_N(R_j,d)$ (Section \ref{ratscrollsec}) and $\Grass_H(R_j,d)$, intersect properly on $\Grass_\tau (R_j,d)$. Do they intersect transversely?  Apply \cite{HN}?
\end{quest}
\begin{quest}Given an Artin algebra $A=R/I$ and a linear element $\ell\in R_1$ there is a partition $\pi_{\ell ,A}$ giving the Jordan type of multiplication by $\ell$ in $A$ \cite{HW,BIK}. The set $\pi(A)=\{\pi_{\ell ,A} \mid \ell \in R_1\}$. How does $\pi (A), A=R/(V)$ behave as $V
$ deforms in $\Grass(R_j,d)$? 
\end{quest}

\begin{ack} This note was inspired by C. Polini's talk at ALGA 2012, at I.M.P.A. in Rio de Janeiro \cite{Po}. It is related to work begun some time ago with F.~Ghione and G.~Sacchiero \cite{GIS}, that depended on Section 4 of \cite{I} and an early version of \cite{I1} and I am grateful for many enjoyable conversations with them in our common language, le fran\c{c}ais. I am grateful for the work of the organizers of ALGA~2012, the welcome of E.~Esteves and his colleagues and the staff of I.M.P.A., the talks and many discussions we had there. A conversation with C. Polini and B. Ulrich while visiting MSRI for the workshop ``Combinatorial Commutative Algebra and Applications'' in December 2012 suggested that a clearer translation to the degree $D$ might be of use. D. Cox pointed out the relation to papers of he and coauthors on parametrization of planar rational curves, discussed in Section \ref{paramratcurves} and made helpful comments. I appreciate the encouragement of E. Esteves and others to write this note. I have had many enlightening conversations with Steve Kleiman over the years, beginning with attending the Zariski seminar and joining the AMS Algebraic Geometry meeting at Woods Hole in 1964 where he led a ``miniseminar''  and welcome the chance to thank him.
\end{ack}
\par


\begin{thebibliography}{ABCDE}

\bibitem[Asc1]{Asc1}
M.-G. Ascenzi: \emph{The restricted tangent bundle of a rational curve on a quadric in $\mathbb P^3$},
Proc. Amer. Math. Soc. 98 (1986), no. 4, 561--566. 

\bibitem[Asc2]{Asc2}
M.-G. Ascenzi: \emph{The restricted tangent bundle of a rational curve in $\mathbb P^2$}, Comm. Alg. 16 (1988), 2193--2208.

\bibitem[BR]{BR}
E. Ballico and L. Ramella: \emph{The restricted tangent bundle of smooth curves in Grassmanians and curves in flag varieties}, Rocky Mountain J. Math 30 (2000),
no. 4, 1207--1227.

\bibitem[BIK]{BIK}
R. Basili, A. Iarrobino, and L. Khatami: \emph{Commuting nilpotent matrices and Artinian Algebras, J. Commutative Algebra (2) \#3 (2010) 295--325}.

\bibitem[Ber1]{Ber1} A. Bernardi: \emph{Normal bundle of rational curves and Waring decomposition}, ArXiv 1203.4955, 2012.

\bibitem[Ber2]{Ber} A. Bernardi: \emph{Apolar ideal and normal bundle of rational curves}, ArXiv 1203.4972, 2012.

\bibitem[Cl]{Cl}
H. Clemens: \emph{On rational curves in $n$-space with given normal bundle},
in Advances in algebraic geometry motivated by physics [Lowell,MA,2000],
137--144, Contemp. Math 276, Amer. Math. Soc., Providence, RI, 2001.

\bibitem[CKPU]{CKPU}
D. Cox, A. Kustin, C. Polini, and B. Ulrich: \emph{A study of singularities on rational curves via syzygies}.  Memoirs AMS 222 (2013), Amer. Math. Soc., Providence, RI. ArXiv 1102.5072.

\bibitem[CSC]{CSC}
D. Cox, T. Sederberg, F. Chen: \emph{The moving line ideal basis of planar rational curves}, Comput. Aided Geom. Design 15 (1998), no. 8, 803--827.

\bibitem[D]{D}
C. D'Andrea: \emph{On the structure of mu-classes}, Communications in 
Algebra 32 (2004), 159--165.

\bibitem[EV1]{EV1}
D. Eisenbud and A. Van de Ven : \emph{On the normal bundle of smooth rational
space curves}, Math. Ann. 256 (1981), 453--463.

\bibitem[EV2]{EV2}
D. Eisenbud and A. Van de Ven: \emph{On the variety of smooth rational space curves with
given degree and normal bundle} Invent. Math. 67 (1982), no. 1, 89--100.

\bibitem[G]{G}
F. Ghione : \emph{Fibres Projettivi}, Preprint \# 11 (1985), Il Univ. degli
Studi di Roma.

\bibitem[GS]{GS}
F. Ghione and G. Sacchiero: \emph{Normal bundles of rational curves in $\mathbb
P^3$}, Manuscripta Math. 33 (1980), 111--126.

\bibitem[GIS]{GIS}
F. Ghione,  A. Iarrobino, and G. Sacchiero: \emph{Restricted tangent bundles of rational curves in $\mathbb P^r$}, preprint, 1988, revised 2002 , 16 p. (still in progress).

\bibitem[GHI]{GHI}
A. Gimigliano, B. Harbourne and M. Ida: \emph{On plane rational curves and the splitting of the tangent bundle},  Ann. Sc. Norm. Super. Pisa Cl. Sci. (5) 12 (2013), no. 3, 587--621. 

\bibitem[HN]{HN}
G. Harder and M. Narasimham: \emph{On the cohomology groups of moduli spaces},
Math. Ann. 212 (1975), 215--248.

\bibitem[HW]{HW}
T. Harima and J. Watanabe: \emph{The commutator algebra of a nilpotent matrix and an application to the
theory of commutative Artinian algebras},  J. Algebra 319 (2008), no. 6, 2545--2570.

\bibitem[He]{He}
G. Hein : \emph{Curves in $\mathbb P^3$ with good restriction of the tangent
bundle}, Rocky Mountain J. Math. 30 (2000), no. 1, 217--235.

\bibitem[HeK]{HK}
G. Hein and H. Kurke: \emph{Restricted tangent bundle on space curves},
Proceedings of the Hirzebruch 65 Conference on Algebraic Geometry (Ramat Gan.
1993), 283--294, Israel Math. Conf. Proc., 9, Bar-Ilan Univ., Ramat Gan, 1996.


\bibitem[I1]{I}
A. Iarrobino: \emph{Punctual Hilbert schemes} A.M.S. Memoir Vol 10, \#188, 1978.

\bibitem [I2]{I1}
A. Iarrobino: \emph{Ancestor ideals of a vector space of forms}, J. Algebra 272 (2004), 530--580.

\bibitem [Ka]{Ka}
H. Kaji: \emph{On the normal bundles of rational space curves},
Math. Ann. 273 (1985), 163--176.

\bibitem [KPU]{KPU}
A. Kustin, C. Polini, and B. Ulrich: \emph{The bi-graded structure of Symmetric Algebras with applications to Rees rings}, preprint (2013) ArXiv 1301.7106.

\bibitem[Mac]{Mac}
F. H. S. Macaulay: \emph{On a method for dealing with the intersection of two plane curves},
Trans. A.M.S. 5  (1904), 385--400.

\bibitem[PS]{PS}
R. Piene and G. Sacchiero: \emph{Duality for rational normal scrolls}, Comm. in
Algebra 12(9) (1984), 1041--1066.

\bibitem[Pol]{Po}
C. Polini, \emph{Rees algebras and singularities}, Talk at ALGA, 12-th Brazilian Meeting on Algebraic Geometry and Commutative Algebra at I.M.P.A., Rio de Janeiro, August, 2012. Video at
http://video.impa.br/index.php?page=12th-alga-meeting.

\bibitem[Ra1]{Ra1}
L. Ramella : \emph{La stratification du sch\'{e}ma de Hilbert des courbes
rationelles de $\mathbb P^n$ par le fibr\'{e} tangent restreint}, C. R. Acad.
Sci. Paris S\'{e}r. I Math. 311 (1990), no. 3, 181--184.

\bibitem[Ra2]{Ra2}
L. Ramella : \emph{Sur les sch\'{e}mas d\'{e}finissant les courbes rationnelles
lisses de $\mathbb P^3$ ayant fibr\'{e} normal et fibr\'{e} tangent restreint
fix\'{e}s}, M\'{e}m. Soc. Math. France (N.S.) No. 54 (1993), ii+74 pp.

\bibitem[Ra3]{Ra3}
L. Ramella: \emph{Strata of smooth space curves having unstable normal bundle},
Boll. Unione Mat. Ital. Sez. B Artic. Ric. Mat. (8) 2 (1999), no. 3, 499--516.

\bibitem[Ran1]{Ran1}
Z. Ran : \emph{The degree of the divisor of jumping rational curves},
Q. J. Math. 52 (2001), no. 3, 367--383.

\bibitem[Ran2]{Ran2}
Z. Ran : \emph{Normal bundles of rational curves in projective spaces},
Asian J. Math. 11 (2007), no. 4, 567--608.

\bibitem[SC]{SC}
X. Shi and F. Chen: \emph{Computing the singularities of rational space curves}, ISSAC 2010 -- Proceedings of the 2010 International Symposium on Symbolic and Algebraic Computation, 171--178, ACM, New York, 2010.

\bibitem[WJG]{WJG}
H. Wang, X. Jia, R. Goldman: \emph{Axial moving planes and singularities of rational space curves}, Comput. Aided Geom. Design 26 (2009), 300--316.
\renewcommand{\baselinestretch}{1.7}

\end{thebibliography}
\end{document}